\documentclass[final,12pt,times,english,twoside,reqno]{amsart}
\usepackage{amsmath,amssymb,mathrsfs,color}
\usepackage{verbatim,setspace,enumerate}
\usepackage{soul,cite,cancel,ulem,graphicx}
\usepackage[modulo]{lineno}
\modulolinenumbers[2]

\usepackage{hyperref,enumitem}
\usepackage{mathtools}
\usepackage{geometry}
\usepackage[utf8]{inputenc}
\usepackage[T1]{fontenc}
\let\strokel\l

\newtheorem{theorem}{Theorem}[section]

\newtheorem{remark}[theorem]{Remark}
\newtheorem{definition}[theorem]{Definition}

\newtheorem{maintheorem}{Theorem}

\newcommand{\Addresses}{{
  \bigskip
  \footnotesize
$^{1}$ Department of Mathematics, University of Palermo,  
Via Archirafi 34, 90123 Palermo, ITALY,
email:dipiazza@unipa.it; Orcid ID: 0000-0002-9283-5157
\\
$^{2}$   Department of Mathematics and Computer Sciences,
     University of  Perugia,  Via Vanvitelli  1, 06123 Perugia, ITALY,
     email: anna.sambucini@unipg.it; Orcid ID: 0000-0003-0161-8729	
}}

\title[Selections of integrable multifunctions ...]{
Selections of integrable multifunctions in  arbitrary Banach spaces}
 \subjclass[2020]{28B20, 28C15, 49J53}
 \keywords{gauge, integrable selection, multivalued integration, Aumann integral}
\author{ Luisa Di Piazza$^{1}$  and Anna Rita Sambucini$^{2}$}

\begin{document}
\begin{abstract}
We find the origin of the integration theory for multifunctions in the sixties in the pioneering works of G. Debreu and R. Aumann, Nobel prizes for the Economy in 1983 and in 2005, respectively.
The Aumann integral 
is defined by means the integrals of measurable selections of the multifunctions. An important tool for the existence of measurable selections is the
Kuratowski and Ryll-Nardzewski theorem, although this famous result needs the separability of the range space.
Other definitions of multifunction integrals, that are not based on selections have been developed
such as Pettis, Henstock, McShane, Birkhoff, Kurzweill-Henstock-Pettis, or variationally integrals.
However  to obtain good properties of such  integrals for multifunctions 
one also needs to study   "nice" properties of selections.

This chapter is devoted to  selection results in the framework 
of an arbitrary Banach space
(non necessarily separable).
In particular we address our attention to the existence of selections
integrable of the same sense of the given multifunction
 and to their applications to the problem of 
representing a multifunction as a translation of a multifunction with better integrable properties by means of
one of its integrable selections.
\end{abstract}
\date{\today}


\maketitle

\tableofcontents

\section{Introduction: a historical point of view}

 The multivalued calculus
 has become an important tool in several areas: for example in control theory, in differential inclusions, 
in economic analysis when equilibria problems are faced 
 \cite{A2005,CCS,Mesquita2025,C2004,G2022,F2022,Satco2008,Satco2013,Sik-2007,PMS,H,M2025}. 
In this last topic the use of the theory of multifunctions has appeared for the first time in the pioneering works  of G. Debreu and R. J. 
Aumann  \cite{D1965,A1965}. The Debreu integral is in some way a multivalued version of the Bochner integral and is defined by means of a suitable embedding of the range space into a Banach space. The Aumann integral is defined via the integrability of measurable selections.
It always exists but it could be empty.

A most applicable theorem on the existence of “good” selections
is the celebrated Kuratowski and Ryll-Nardzewski  Theorem  \cite{KRN}, which holds in a separable space. This is also one of the reasons  why
in the next three decades  the multivalued integration theory was mainly concerned  with multifunctions whose
values are subsets of a separable Banach space. 

 Even though the definition of Aumann integrability
 is 
 very natural and  useful in applications, it lacks several good properties.
 Therefore, after this work, various notions of integrals for multivalued maps
 have been developed using different
 techniques,  see for example 
\cite{bcs2014,bmis,BS1,CCGS,ccgis,CR,danilo,cgis2023,cgis2022,CG,DPM2014,EH2000,DPM2005,K2025,Kalita2024}
 that are not based on selections. 
  However to obtain 
good properties for applications  of that types of integral  some
selection results are needed. 
 In this chapter we examine the state of the art in the  non-separable case 
 \cite{CKR2009,CKR2010, CDPMS1,CDPMS2}
and we limit ourselves in particular to considering the case of the  Pettis and the Pettis type integrals 
\cite{C1,DPM2005,DPM2013,DPM2009,EH2000,K2021,M2011,Mu2025,C2001}
 and the gauge type integrals
  \cite{CDPMS1,CDPMS2,Kalita2025,BS1,bcs2014,BM,bmis,CCGS,ccgis,CDPMS3,CDPMS4,CS1,cr2005,CG,DPM2006a}.
Many of the results mentioned above were obtained  in fruitful scientific collaboration by the authors with A. Boccuto,   B. Bongiorno, D. Candeloro, V. Marraffa and  K. Musia\strokel.
 \\
Let us take a small step back. 
In the scalar case
the gauge integral was introduced as neither the Riemann nor the Lebesgue integrals fully address the existence of primitives or the issue of differentiating under the integral sign.

 Denjoy (1912)  and Perron (1914) indipendently developed an integral that is more general than the Lebesgue one and extends the fundamental theorem of calculus, resolving these problems. 
 However, the methods they employed were quite complex.
 In the 1960s Kurzweil and Henstock, indipendently, proposed a new definition
 (equivalent to those of  Denjoy and Perron)
 powerful to include every  derivative and with  the advantage
  that this technique is simple and requires only small modifications with respect to the Riemann one.

As mentioned earlier, the construction of the gauge integrals closely resembles that of the Riemann integral, with one key difference: instead of using the mesh size to measure the fineness of a tagged partition, 
 a gauge $\gamma$ is used, which need not be uniform in the domain 
 $\Omega$ of the function $f$.
The idea is to choose the gauge $\gamma$ according to the function $f$ that we want to integrate:
where f is badly-behaved, $\gamma$ is chosen to be small. 
The condition {\it whenever P is a $\gamma$-fine
tagged partition of\, $\Omega$}  never vacuous is very important:
its proof basically relies on  a compactness argument and was given in $\Omega = [0,1]$ first in 1895 by 
Cousin who was a student of Poincaré.  This result was rediscovered by Lebesgue  in his  thesis in 1903 and it was known 
for a long time as the Borel-Lebesgue Theorem.
It is now occasionally referred to as Thomson's Lemma, as he extended it to encompass full covers.\\
The Henstock–Kurzweil definition of the gauge integral is now easy to state: it is just the Riemann integral with “mesh $\delta$” replaced by “$\gamma$-fine”.
 To integrate a function $f$ taking values in a Banach space $(X, \Vert \cdot \Vert)$ to obtain an integral value $v \in  X$, just replace absolute values by norms 
\cite{DPM2013,DPM2009,DPM2006,DPM2006a,DPM2005,CS1,M2025,cmn,K,K2002}.
\\
Similarly, the integrals of McShane and  of Birkhoff are defined; in the former it is not required that the tags belong to the corresponding set,
 in the second case  unconditionally convergent series are considered 
 \cite{F,dp-p,DM,Po,Po1,M,FMNR,cr2005,CR,CG,Bir,BM,BS1,R1,R2,R3,F1994}.
 This is actually  the formulation of the Birkhoff integral given by Fremlin  \cite{F1994} and it can be considered a gauge integral because Naralenkov and Solodov \cite{nara,solo}
  independently showed that the Birkhoff integral is equivalent to that of McShane with measurable gauges.
When the target space is a separable Banach space, Pettis, McShane and Birkhoff are all the same.
Additionally,  a variational version of the gauge integral can be considered \cite{Po1,dmm16,BDpM2,K2002,CDPMS1,CDPMS2,CDPMS3}, for example to apply them to differential inclusions,  \cite{CCS,G2022,Mesquita2025,F2022,C2004}. 
 For a comprehensive treatment of these applications see also  the Special Issue of  Czechoslovak Mathematical Journal  \cite{T2025},  dedicated to  Jaroslav Kurzweil and the references therein.
In the scalar case the variational integrals   were introduced by Thomson and Pfeffer  to generate Borel measures in $\mathbb{R}$.\\
\section{Definitions and basic facts}\label{due}
Let  $(\Omega, \Sigma, \mu)$ be a complete, finite measure space  and $X$ be a Banach space  with  dual $X^*$ whose 
 closed unit ball  is denoted by  $B_{X^*}$. 
We denote by  $cwk(X)$  the family of  all convex and non-empty, weakly compact subsets of $X$ and by $ck(X)$ the family of all  compact members of $cwk(X)$ and  $d_H$ is the Hausdorff distance in $cwk(X)$.\\

     For every $C \in cwk(X)$ the {\it support
  function of} $C$ is denoted by $s( \cdot, C)$ and defined on $X^*$ by 
  $s(x^*, C) = \sup \{ \langle x^*,x \rangle : \ x  \in C\}$. 
  For all unexplained definitions we refer to the book of Castaing and Valadier \cite{CV} or the book of Hu and Papageorgiou\cite{hp}.\\

  A map $\Gamma: \Omega \to 2^X\setminus\{\emptyset\}$
 is called a {\it multifunction}. 
  $\Gamma$ is said to be

 \begin{description}
 \item[$m_1)$] {\it Effros measurable} (or simply {\it measurable})  if for each open set $O \subset X$, the set 
 $ \{t\in \Omega: \Gamma(t)\cap O \neq \emptyset\} \in \Sigma$.
\item[$m_2)$] 
 {\it scalarly measurable} if for every $ x^* \in X^*$, the map
  $s(x^*,\Gamma(\cdot))$ is measurable. 
\item[$m_3)$]  {\it Bochner measurable}  if there exists a sequence of simple multifunctions $(\Gamma_n)_n$ such that
$\lim_{n\rightarrow \infty}d_H(\Gamma_n(t),\Gamma(t))=0$, a.e. in $ \Omega$,
where a multifunction is simple if  there exist $E_1,..., E_k $ pairwise disjoint measurable subsets of $\Omega$   with
 $\cup_{j=1}^k E_j = \Omega$  and 
  $\Gamma$ is constant on each $E_j$.
\end{description}
It is well known that the $m_1)$ 
assumption
of a $cwk(X)$-valued  multifunction
 yields   $m_2)$ (if $X$ is separable  the reverse implication is also true). 
 Moreover, $m_3)$ implies $m_1)$  (see \cite{hp}). The reverse implication fails (see \cite[Example 3.8]{CDPMS1}).\\

\noindent Given a multifunction $\Gamma$, a
 function $f:\Omega\to X$ is called a {\it selection of } $\Gamma$ if $f(t) \in\Gamma(t)$,   for every $t\in \Omega$.\\
Given a  measurable multifunction $\Gamma$ 
 we recall that  its  {\it Aumann integral}, see \cite{A1965},  is the set: 
\begin{eqnarray*}
{\scriptstyle(A)}\int_{\Omega} \Gamma d\mu := \Big\{ \int_{\Omega} f d\mu: \quad f \in L^1(\mu)\,\,\, \& \,\,\, f \mbox{ is a selection  of } \Gamma \Big\}.
\end{eqnarray*}

 The great interest in the existence of measurable selections was motivated for example by
 the applications in differential inclusions,  
  control theory, mathematical models for Economy
(see  \cite{C2004,F2022,G2022,Mesquita2025,Satco2008,Satco2013,Sik-2007,H}
 and the references therein) 
   and multivalued integration.  As said in the Introduction, 
the Aumann integral always exists but it could be empty. 
A keystones result for the existence of measurable selections for a multivalued function  is the 
Kuratowski and Ryll-Nardzewski's Theorem \cite{KRN}, \cite[Theorem III.30]{CV}
 which guarantees the existence of measurable selections for 
measurable multifunctions.
The drawback of this results is   the requirement of the
separability of the range  space. 
In 2009 Cascales, Kadets and Rodrigues 
\cite{CKR2009,CKR2010}
 have been removed the separability contrains of the range space.
In fact they proved the following beautiful result: 
{\it 
each scalarly measurable multifunction taking values in the family of weakly compact (non necessarily convex) sets of a general Banach space 
possesses scalarly measurable selections}.
For other details on measurable selections 
we refer the reader to \cite{CV,H}.
 
   The Aumann integral  is devoid  of good properties as closedness, convexity, compactness and convergence under a  sign integral.   For this reason
    different notions of integrals for multifunctions have been developed.
    The first one was the Pettis multivalued integral introduced first by Castaing and its school \cite{C1,C2,EH2000,CV,Z1996,Z2000}. 
    In Section \ref{tre} we summarize some of the basic results on this multivalued integral, while 
   in the subsequent sections we provide  an overview of the selection results known in the literature for Pettis-type  and  gauge multivalued integrals.


   \section{The multivalued Pettis integral}\label{tre}
  \begin{definition} \rm  A multifunction $\Gamma:\Omega\to cwk(X)\,  $ is said to be
 {\it  Pettis integrable}  in $cwk(X) \,  $ if
 \begin{description} 	
\item[\ref{tre}.a)]  the function $s(x^*,\Gamma(\cdot))$ is Lebesgue integrable on $\Omega$ for each	$x^* \in X^*$ ;
\item[\ref{tre}.b)]   for each  $A \in \Sigma$
 		 a set $M_{\Gamma}(A)\in cwk(X) $ exists  such that 
 	      \begin{eqnarray}\label{Pettis}
 	      s(x^*,M_{\Gamma}(A))=\int_As(x^*,\Gamma(t) )\,dt,  \quad \forall\,\, x^* \in X^*.
 	      \end{eqnarray}
 		We set  ${\scriptstyle (P)}\displaystyle{\int_A}\Gamma \,d\mu := M_{\Gamma}(A)$.
 	 \end{description}
 \end{definition} 
 
 When the range space $X$ is separable 
 we have the following 
 summarization of  different results given by H. Ziat, K. El Amri and  C. Hess, C. Cascales, V. Kadets and J. Rodr\'iguez.

	\begin{maintheorem}{\rm (\cite{EH2000, Z1996,Z2000}, \cite[Chap. V]{CV}, \cite[Theorem A]{CKR2009})}
\label{compilation}	
	Let $X$ be a    separable Banach space and let $\Gamma:\Omega\to {cwk(X)}$ be a  multifunction.   The following  conditions are equivalent:
	\begin{description}
	\item[\ref{compilation}.1)] $\Gamma$ is  Pettis integrable in ${cwk(X)}$;
	\item[\ref{compilation}.2)] the family $W_{\Gamma}=\left\{s(x^*, \Gamma (\cdot)): \ x^* \in B_{X^*} \right\}$ is uniformly integrable; 	
	\item[\ref{compilation}.3)] the family $W_{\Gamma}$ is consisting  of measurable functions and all  scalarly measurable selections of  \, $\Gamma$ are Pettis integrable.	
\end{description}	
	In this case, for each  $A \in \Sigma$
	 \begin{eqnarray}\label{eq-P}
	{\scriptstyle  (P)}\int_A \Gamma d\mu =\left\{ {\scriptstyle (P)}\int_A f d\mu:  \ \ f \in S_P(\Gamma)   \right\}, 
	 \end{eqnarray}	
	 where $S_P(\Gamma)$ denotes the family of Pettis integrable selections of \, $\Gamma$. 	
	 \end{maintheorem} 
		If $X$ is a  general Banach space   we have
		
	\begin{maintheorem}{\rm (\cite[Corollary 2.3 and Theorem 4.2 ]{CKR2009}, \cite[Theorem 4.1]{CKR2009})}
\label{compilationb}	
	Let $X$ be  any  Banach space and let $\Gamma:\Omega\to {cwk(X)}$ be a scalar measurable   multifunction.    Then  {\bf \ref{compilation}.1)} $ \Leftrightarrow$  {\bf \ref{compilation}.3)}  and 
		 {\bf  \ref{compilation}.1)}  $ \Rightarrow$ {\bf \ref{compilation}.2)}. 
	 \end{maintheorem} 		
				 
		Moreover, in an arbitrary Banach space,  according to \cite[Theorem 2.6]{CKR2009}, the equality \eqref{eq-P} does not hold in general, and we have  the following:
		\begin{eqnarray}\label{PHK}
		{\scriptstyle  (P)}\int_A \Gamma d\mu =
		\overline{ \left\{ {\scriptstyle (P)}\int_A f d\mu:  \ \ f \in S_P(\Gamma)   \right\}}, \quad \mbox{for every } A \in \Sigma.		
		\end{eqnarray}
				

\section{The Henstock-Kurzweil-Pettis integral}\label{quattro}
A straightforward extension of the Pettis integral for multifunctions can be achieved by substituting the Lebesgue integrability of the support functions $s(x^*,\Gamma)$  with  
a weaker integrability, for example
 the Henstock-Kurzweil integrability.
This produces the Henstock-Kurzweil-Pettis integral that, as
we shall see, can be described in terms of 
the Pettis set-valued integral.\\

From now on we consider  (multi)functions defined in the unit interval of $\mathbb{R}$, equipped with the usual topology and the Lebesgue measure $\lambda$.
$\mathscr{L}$ is the family of all Lebesgue measurable subsets of $[0,1]$ and $\mathscr{I}$ denotes the collection of closed subintervals of $[0,1]$.
If $I \in \mathscr{I}$ then $\vert I \vert$ is its lenght.
\\

To introduce the Henstock-Kurzweil integral for single valued functions 
we recall that
\begin{enumerate}[label=\alph*)]
\item 
a {\it partition} ${\mathscr P}$ {\it in} $[0,1]$ is a collection $\{(I_1,t_1),$ $ \dots,(I_p,t_p) \}$,
  where $I_i$ are nonoverlapping subintervals of $[0,1]$, $t_i \in [0,1]$, $i=1,\dots,p$.
 If $\cup^p_{i=1}I_i=[0,1]$, then  ${\mathscr P}$ is {\it a partition of } $[0,1]$.
 \item  If   $t_i  \in I_i$, $i=1,\dots,p$,  we say that $P$ is a {\it Perron partition of} $[0,1]$.
\item  A {\it gauge} $\gamma$ on $[0,1]$ is a positive function on $[0,1]$. For a given gauge $\gamma$ on $[0,1]$,
  we say that a partition $\{(I_1,t_1), \dots,(I_p,t_p)\}$ is $\gamma$-{\it fine} if for  $i=1,\dots,p$, it is
  $I_i\subset(t_i-\gamma(t_i),t_i+\gamma(t_i))$.
  \end{enumerate}
  $\Pi_{\gamma}$ and $\Pi_{\gamma}^P$ are  the families of $\gamma$-fine partitions, and $\gamma$-fine Perron partitions of $[0,1]$, respectively.

 \begin{definition}\label{def-H}
 \rm 
  A function $g:[0,1]\to X$ is said to be
  \begin{description}
  \item[\ref{def-H}.a)] {\it  Henstock}
   integrable on $[0,1]$  if there exists   $x \in X$
    such that for every $\varepsilon > 0$ there exists a gauge $\gamma$ on $[0,1]$
such that for each partition
   $\{(I_1,t_1), \dots,(I_p,t_p)\} \in \Pi_{\gamma}^P$, 
   we have
\begin{eqnarray*}
\Big\Vert x -\sum_{i=1}^p g(t_i)|I_i| \Big\Vert <\varepsilon\,.
\end{eqnarray*}
In this case we set $x := {\scriptstyle (HK)}\int_{[0,1]} g d\lambda.$ If $X=\mathbb{R}$ the integral is named  {\it Henstock-Kurzweil  integral}.
 \item[\ref{def-H}.b)] {\it  Henstock-Kurzweil-Pettis}
   integrable on $[0,1]$,  if
   for every $x^* \in X^*$ the function $x^*g$ is Henstock-Kurzweil integrable on $[0,1]$  and 
    there exists   $x \in X$
    such that, for every  $x^* \in X^*$
\begin{eqnarray*}
x^*(x)=  {\scriptstyle (HK)}\int_{[0,1]} x^*(g)\, dt.
\end{eqnarray*}
In this case we set $x := {\scriptstyle (HKP)}\int_{[0,1]} g \,d\lambda.$
\end{description}
 \end{definition}
 
If $g$ is Henstock-(Kurzweil-Pettis) integrable on $[0,1]$ then it is integrable in the same sense  on every $I\in \mathscr{I}$.  
For other results and the properties of the Henstock-(Kurzweil Pettis) integral we refere to  \cite{DPM2005,DPM2006,DPM2009,DPM2013}.

  \begin{definition}\label{HKPdef} \rm  A multifunction $\Gamma:[0,1]\to cwk(X)$\,  is said to be
 {\it  Henstock-Kurzweil-Pettis integrable}  (briefly HKP-integrable)  in $cwk(X) $  if
 \begin{description} 	
\item[\ref{HKPdef}.a)]  the function $s(x^*,\Gamma(\cdot))$ is Henstock-Kurzweil integrable on $[0,1]$ for each	$x^* \in X^*$;
\item[\ref{HKPdef}.b)]   for each  $I\in \mathscr{I}$
 		 a set $M_{\Gamma}(I)\in cwk(X) $ exists  such that 
 	      \begin{eqnarray}\label{Pettis}
 	      s(x^*,M_{\Gamma}(I))={\scriptstyle (HK)}\int_I s(x^*,\Gamma(t))\,dt,  \quad \forall\,\, x^* \in X^*.
 	      \end{eqnarray}
 		We set  ${\scriptstyle (HKP)}\int_I\Gamma \,d\lambda:=M_{\Gamma}(I)$.
 	 \end{description}
 \end{definition} 
We denote by $\mathscr{S}_{HKP}(\Gamma)$ the family of all Henstock-Kurzweil-Pettis integrable
selections of $\Gamma$  and, for every $I \in \mathscr{I}$,  
\[ 
IS_{\Gamma}(I) := \Big\{ {\scriptstyle (HKP)}\int_{I} f d\lambda: \,\, 
f \in \mathscr{S}_{HKP}(\Gamma) \Big\}.\]

\begin{theorem}\label{i-ii}
{\rm \cite[Theorem 1]{DPM2009}} 
Let  $X$ be any Banach space and let $\Gamma:[0,1]\to cwk(X)$ 
be a  scalarly measurable multifunction. The following are equivalent:

\begin{itemize}
\item[$(i)$] \, 
$\Gamma$ is HKP-integrable  in $cwk(X)$;
\item[$(ii)$] \,   each scalarly measurable selection of $\Gamma$ is HK-integrable.
\end{itemize}

	In this case,
	 for each $I \in \mathscr{I}$	
\[
	{\scriptstyle (HKP)}\int_I \Gamma \ d\lambda:= 
	\overline{\Big\{{\scriptstyle (HKP)}\int_I f d\lambda: f \in
		{\mathscr{S}}_{HKP}(\Gamma)\Big\}}\;. 
\]

\end{theorem}

\begin{proof} We give here only a sketch of the proof.
If $\Gamma$ is  HKP integrable, then 
 for every $x^* \in X^*$, 
$s(x^*,\Gamma(\cdot))$ is
Henstock-Kurzweil integrable and then measurable.
According to \cite[Theorem 3.8]{CKR2010}  scalarly measurable selections  of $\Gamma$ exist. 
Let $g$ be one of such selections.
Then, for each $x^* \in X^*$  we have
\begin{eqnarray*}
&& -s(-x^*,\Gamma(t))\leq x^*g(t)\leq s(x^*,\Gamma(t))\,,
\\ &&
0\leq x^*g(t)+s(-x^*,\Gamma(t))\leq s(x^*,\Gamma(t))+s(-x^*,\Gamma(t))\,.
\end{eqnarray*}
and the
Henstock-Kurzweil integrability of the function $x^*g$ follows.\\
	Moreover for  each    $I \in \mathscr{I}$
\[	-s(-x^*,M_{\Gamma}(I))\,\leq (HK)\int_I x^*g(t)\,dt\leq
	s(x^*,M_{\Gamma}(I)))\,.
\]
	Since $M_{\Gamma}(I)$ is convex and weakly compact,
	it follows that 
	 $g$ is Henstock-Kurzweil-Pettis-integrable. \\
	Viceversa for every $I \in \mathscr{I}$ the closure of $IS_{\Gamma}(I)$ is convex. We  prove that 
$\overline{IS_{\Gamma}(I)}$	
 is the desired value of the integral. For every $x^* \in X^*$ let 
	\begin{eqnarray}\label{G}
G(t) = : \big\{ x \in \Gamma(t): x^*(x) = s(x^*, \Gamma(t))\big\}.
\end{eqnarray}
	By \cite[Lemma 2.4]{CKR2009} $G$ is scalarly measurable and let $g$ be  a scalarly measurable selection of $G$ and so of $\Gamma$.
	 Hence $g$ is HKP-integrable and 
	$s(x^*, \Gamma)= x^* g$. \\
	By definition
	${\scriptstyle (HKP)}\int_{I} g d\lambda \in IS_{\Gamma}(I)$. 
Moreover, for every $f \in  \mathscr{S}_{HKP}(\Gamma)$, it is
\[x^*\Big ( {\scriptstyle (HKP)}\int_{I} g d\lambda \Big) \geq  x^*\Big ( {\scriptstyle (HKP)}\int_{I} f d\lambda \Big)\]	
and then 
\[x^*\Big ( {\scriptstyle (HKP)}\int_{I} g d\lambda \Big) = \sup_{ x \in IS_{\Gamma}(I)} x^*(x) = s(x^*, \overline{ IS_{\Gamma}(I)}).\]	
This proves, by  James's characterization of weak compactness, that 	$\overline{ IS_{\Gamma}(I)} \in cwk(X)$  and the Henstock-Kurweil-Pettis integrability of $\Gamma$ since
\[ s(x^*, \overline{ IS_{\Gamma}(I)}) = {\scriptstyle (HK)}\int_{I} s(x^*,\Gamma(t)) dt.\]	
\end{proof}
As an application of the existence of HKP integrable selections we obtain that a HKP integrable multifunction is a translation of a Pettis integrable multifunctions by means one of its HKP integrable selections.
More precisely

\begin{theorem}
{\rm \cite[Theorem 1]{DPM2009}} \label{i-iii}
Let  $X$ be any Banach space and let $\Gamma:[0,1]\to cwk(X)$
 be a  scalarly measurable multifunction. The following are equivalent:

\begin{itemize}
\item[$(i)$] \, 
$\Gamma$ is HKP-integrable  in $cwk(X)$;
\item[$(iii)$] \, 
${\mathscr{S}}_{HKP}(\Gamma)\not=\emptyset$   and  for every
$f \in {\mathscr{S}}_{HKP}(\Gamma)$
the multifunction
$G$ defined by $\Gamma(t)=G(t)+f(t)$,  is  Pettis
integrable in $cwk(X)$.
\end{itemize}
\end{theorem}

\begin{remark} \rm 
In the Definition \ref{HKPdef} we can substitute $cwk(X)$ with $ck(X)$.
In such a  case  in Theorem \ref{i-ii} condition (i) implies condition (ii) and in Theorem \ref{i-iii}
condition (i) is equivalent to condition (iii), \cite[Theorem 2 and Question 1]{DPM2009}. \\
Moreover, by \cite[Lemma 3.1]{CDPMS4} if $\Gamma$ is a $cwk(X)$-valued HKP integrable multifunction, $f$ is one of its selections integrable in the same sense, and 
if we denote their primitives  by $\varPhi$ and $\phi$ respectively,  then we have
\begin{eqnarray}\label{3.1del3}
\phi(I) \in \varPhi(I), \qquad \mbox{ for every  } \,I \in \mathscr{I}.
\end{eqnarray}

\end{remark}
 \section{The gauge multivalued  integrals}\label{cinque} 
The Pettis and the Henstock-Kurzweil-Pettis multivalued integrals were  possible extensions of the Aumann integral, not the only ones.
We  consider now the gauges type integrals using a different approach: not applied to selections but to the whole multifunction $\Gamma$.\\

On $cwk(X)$ the
   Minkowski addition ($A \, \dot{+}\, B :\,=\big\{a+b:a\in A,\,b\in B \big\}$) and the
   standard multiplication by scalars are considered.  As said before $d_H$ is the Hausdorff distance  and
   $\|A\|:=\sup\{\|x\|\colon x\in{A}\}$.
 We refer to \cite{L1} for  properties of $(cwk(X), d_H, \dot{+}, \cdot)$.\\

  In particular
$cwk(X)$ can be embedded
  as a convex cone 
 into the Banach space  $l_{\infty} (B_{X^*})$
 endowed with the supremum norm $\Vert \cdot \Vert_{\infty}$,  by means the map 
$i:cwk(X) \to  l_{\infty} (B_{X^*}), \,\, \mbox{ defined by } \,\, i(C)(x^*)=s( x^*, C)$
(see \cite[Chapter II]{CV}).
Such a map is called in the literature 
 R{\aa}dstr\"{o}m embedding
and it  is isometric and  linear.
Therefore $cwk(X)$ could be considered a ``near vector space'', ( essentially a vector
space without additive inverses). 
\\

It is also possible to consider  other  R{\aa}dstr\"{o}m embeddings, for example in
 the space $C(T)$, where $T$ is a  compact Hausdorff space \cite{L1}.   In any case, the results of the embedding theorems are independent of the target Banach space considered.
 \\

The  map
$i:cwk(X) \to l_{\infty} (B_{X^*})$ satisfies the following properties:
\begin{description}
\item[R.1)] $i(\alpha A+ \beta C) = \alpha i(A) + \beta i(C)$ for every $A,C\in cwk(X), \alpha, \beta \in
\mathbb{R}^+;$
\item[R.2)] $d_H(A,C)=\|i(A)-i(C)\|_{\infty},\ \ A,C\in cwk(X)$;
\item[R.3)] $i(cwk(X))=\overline{i(cwk(X))}\ \ \ \ {\rm (norm\ closure). \ }$
\end{description}

We say that
    \begin{definition}\label{H-MS} \rm
 A multifunction $\Gamma:[0,1]\to cwk(X)$ is said to be:
 \begin{description}
 \item[\ref{H-MS}.1)]  {\it Henstock} (resp. {\it McShane})
   integrable on $[0,1]$,  if there exists   $\varPhi_{\Gamma}([0,1]) \in cwk(X)$
    such that for every $\varepsilon > 0$ there exists a gauge $\gamma$ on $[0,1]$
such that for each partition 
   $\{(I_1,t_1), \dots,(I_p,t_p)\}$ in  $\Pi_{\gamma}^P$   (resp. in $  \Pi_{\gamma}$), we have
\begin{eqnarray}\label{e14}
d_H \left(\varPhi_{\Gamma}([0,1]),\dot{\sum}_{i=1}^p\Gamma(t_i)|I_i|\right)<\varepsilon\,.
\end{eqnarray}
\item[\ref{H-MS}.2)]
 {\it Birkhoff} 
   integrable on $[0,1]$,
   if there exists   $\varPhi_{\Gamma}([0,1]) \in cwk(X)$
 such that: for every $\varepsilon > 0$  there is a countable
partition $\Pi_0$ in  $\mathscr{L}$  of $[0,1]$ such that for every  partition $\Pi = (A_n)_n$
 of $[0,1]$  in $\mathscr{L}$
finer than $\Pi_0$ and any choice $ (t_n)_n$, $t_n \in A_n$, the series
$\dot{\sum}_n \lambda(A_n) \Gamma(t_n)$
 is unconditionally convergent  
 and
\begin{eqnarray}\label{e14-a}
d_H \big(\varPhi_{\Gamma}([0,1]),\dot{\sum}_n \Gamma(t_n)\lambda(A_n)\big)<\varepsilon\,.
\end{eqnarray}
\end{description}
\end{definition}

\begin{remark}\label{nota} \rm 
\phantom{a}
\begin{enumerate}[label=\alph*)]
\item A multifunction $\Gamma:[0,1]\to cwk(X)$ is said to be {\it Henstock} (resp. {\it McShane or Birkhoff})
   integrable on $I\in\mathscr{I}$ if $\Gamma 1_I$ is respectively integrable on $[0,1]$. \\
   We then write  ${\scriptstyle (H)}\int_I \Gamma\,dt:=\varPhi_{\Gamma 1_I}([0,1])$ (resp.  
   ${\scriptstyle (M)}\int_I\Gamma\,dt:=\varPhi_{\Gamma 1_I}([0,1]$) or 
   ${\scriptstyle (Bi)}\int_I\Gamma\,dt:=\varPhi_{\Gamma 1_I}([0,1]$)).
 It is known that a multifunction that is Henstock (McShane, Birkhoff) integrable on $[0,1]$ is integrable in the same manner 
  on each $I\in \mathscr{I}$ (see e.g. \cite{CDPMS1}).

From the definition, it 
follows at once 
 that if \( \Gamma \) is McShane  integrable, then it is also Henstock  integrable, yielding the same integral values.  
\item According to H\"{o}rmander's equality \cite{CV}
\[\phantom{AaA}  d_H \Big(\varPhi_{\Gamma}([0,1]),\dot{\sum}_{j=1}^p \Gamma(t_j)|I_j|\Big) = 
\sup_{x^* \in B_{X^*}}  \Big\vert s(x^*, \varPhi_{\Gamma}([0,1])) - 
\sum_{j=1}^p s(x^*,\Gamma(t_j) |I_j| ) \Big\vert \]
  the embeddings $i$    allow to 
    reduce  the  gauge integrability of multifunctions  to  the    gauge integrability of the functions  $i \circ \Gamma$. 
In fact,  if $z=\int i\circ\Gamma\, d\lambda \in l_{\infty} (B_{X^*})$, then $K\in{cwk(X)}$
 exists   with $i(K)=z$.\\
So it is possible to introduce the gauge integrals as in the single-valued case 
and a multifunction $\Gamma:[0,1]\to cwk(X)$  is  Henstock  (resp. McShane)
   integrable  if and only if the single valued function $i\circ \Gamma$ is Henstock    (resp. McShane) integrable in the usual sense.
\item 
For the Birkhoff case it is possible to do the same.  
The Birkhoff integral can be considered    a gauge integral.
 In fact  
a vector valued function $g:[0,1] \to Y$, where $Y$ is a Banach space,  is Birkhoff-integrable if and only if
there exists an element $y\in Y$ such that
for each $\varepsilon>0$ a {\it measurable} gauge $\gamma$ can be found on $[0,1]$, such that,
as soon as ${\mathscr P}=\{(t_j,I_j):j=1,...,p\}$ is any $\gamma$-fine
partition of $[0,1]$, it holds
$\| \sum_{i=1}^p|I_i|g(t_i)  \, - \, y \| < \varepsilon$.
\\ 
For the equivalence of this definition with the more
common notion of Birkhoff integrability see \cite{F,F1994,Po,Po1,CCGS,R1,R2,R3,ccgis,nara,solo}.\\
There is a large literature on $X$-valued gauge integrals, we quote here
 \cite{cmn,cr2005,CG,DM,FMNR,F,F1994,K2002,K,M,Po,Po1,nara,Me}.\\

\item 
The embedding construction cannot be applied  to the Pettis integrability of multifunction directly. In fact the Pettis integrability of 
$\Gamma$ and $i \circ \Gamma$ are not equivalent. See \cite[Proposition 4.5]{CKR2009} for a sufficient condition
under which they are equivalent.
 In general, according to  \cite[Proposition 4.4]{CKR2009},  if $i \circ \Gamma$ is Pettis integrable then also $\Gamma$ is Pettis integrable and 
\[ i \circ {\scriptstyle  (P)}\int_A \Gamma\, d\mu = {\scriptstyle  (P)}\int_A  i \circ \Gamma\, d \mu, \quad \forall \, A \in \Sigma.
\]
\item Any Henstock integrable multifunctions $\Gamma$  is also Henstock-Kurzweil-Pettis integrable,
whereas if $\Gamma$ is McShane integrable then it is also Pettis integrable, 
 since $i \circ \Gamma$ is McShane and then Pettis integrable.
To show   multifunctions that are Pettis integrable but non McShane integrable (resp.  Henstock integrable) we quote \cite[Example 2.10]{CDPMS3} (resp.  \cite[Example 2.12]{CDPMS3}); 
whereas in  \cite[Example 2.11]{CDPMS3} a McShane integrable but non Birkoff integrable multifunction was given. 
\item When $X$ is an arbitrary Banach space, each scalarly measurable selection of a Pettis (resp. Henstock-Kurzweil-Pettis)
 integrable multifunction $\Gamma$ is also Pettis (resp. Henstock-Kurzweil-Pettis) integrable.
 For the gauge integrals  the behavior is different. 
\end{enumerate}
\end{remark}

If 
 we denote by $\mathscr{S}_H(\Gamma), \mathscr{S}_{MS}(\Gamma)$ and $ \mathscr{S}_{Bi}(\Gamma)$ the families of all scalarly measurable selections of $\Gamma$
 that are integrable in the sense of Henstock, McShane, Birkoff respectively, we have the following:
\begin{theorem}\label{3.1di DPM2014}
{\rm \cite[Theorem 3.1]{DPM2014}}
Let $X$ be any Banach space.
If $\Gamma: [0,1] \to cwk(X)$ is Henstock integrable in $cwk(X)$  then $\mathscr{S}_H(\Gamma)$ is non-empty.
 If moreover $\Gamma$ is also Pettis integrable, then $\mathscr{S}_{MS}(\Gamma)$ is non-empty.
\end{theorem}
\begin{proof}
We give a sketch of the proof.
Since $\varPhi_{\Gamma}([0,1]) \in cwk(X)$, there exists a strongly exposed point
$x_0 \in \varPhi_{\Gamma}([0,1])$ 
and let $x_0^* \in B_{X^*}$ be such that $x^*_0(x_0) >x^*_0(x)$ for every $x \in \varPhi_{\Gamma}([0,1]) \setminus \{ x_0\}$.\\
 Let  $G$ as in formula \eqref{G} of Theorem \ref{i-ii} for $x_0^*$. By \cite[Proposition 2]{DPM2009}
  $G$ is HKP integrable  and Pettis integrable if $\Gamma$  is.\\
  The selection $g$ constructed as in Theorem \ref{i-ii} is then  HKP (Pettis) integrable and
\[ x^*_0(x_0) = \int_{[0,1]}  x^*_0 (g) d\lambda.\]
Let  $\gamma$ be  a gauge  suitable for the definitions of the Henstock integrability of $\Gamma$ and the Henstock-Kurzweil integrability of $x^*_0 g$. \\
Because  $x_0$ is a strongly exposed point it is possible to prove that 
$\gamma$ satisfies the definition of Henstock integrability for $g$.  Then the McShane integrability of $g$ follows since $g$ is Pettis if $\Gamma$ is.
\end{proof}

As an application of the existence of Henstock integrable selections for a Henstock integrable multifunction we have
\begin{theorem}
{\rm (\cite[Theorem 3.2]{CDPMS2})}
Let
$X$ be any Banach space and let  $\Gamma: [0,1] \to cwk(X)$ be a multifunction. Then the following are equivalent
\begin{itemize}
\item[$(i)$] $\Gamma$ is Henstock integrable in $cwk(X)$;
\item[$(ii)$] $\mathscr{S}_H (\Gamma) \neq \emptyset$ and for every $f \in \mathscr{S}_H (\Gamma)$ the multifunction $G$ defined by $\Gamma(t) = G(t) + f(t)$, is McShane integrable.
\end{itemize}
\end{theorem}

\begin{remark} \rm \phantom{a}
\begin{enumerate}[label=\alph*)]
\item The existence of Birkhoff integrable selections 
 for a Birkhoff integrable multifunction $\Gamma$
follows (see  \cite[Theorem 3.4]{CDPMS1})
  since $i \circ \Gamma$ is in particular McShane integrable. 
Therefore a  McShane integrable selection $g$ can be obtained and a  measurable  gauge for $g$ can be chosen  since $\Gamma$ is Birkhoff integrable.
\item 
 In general, for the considered gauge integrals,   not every scalarly measurable selection of a given multifunction is integrable
 in the same sense of the multifunction, 
 as proven in \cite[Proposition 3.2]{CDPMS1}.
 There the example works for a variationally McShane integrable multifunction (see Section \ref{s6} below). But according to Remark \ref{rem6.5} below, each variationally McShane integrable multifunction
 is also McShane and Henstock integrable.
Moreover (see \cite[Remark 3.5]{CDPMS1}) the example works also the Birkhoff integral. \\ 
It needs to require additional conditions on the Banach space $X$
or on the target space of the multifunction to obtain 
the integrability of all measurable selections in the same sense of the given multifunction.
In fact,  
 if $X$ is separable and we consider a $ck(X)$-valued multifunction $\Gamma$ 
 or,  if the multifunction $\Gamma$ has an almost separable range and  if $\Gamma$ is    Henstock (resp. McShane) integrable, then each scalar measurable 
 selection of $\Gamma$ is integrable in the same sense (see \cite{DPM2006a,CDPMS1}). 
 
 The same is true  if we consider $cwk(X)$-valued multifunctions taking values in any Banach space X with the property
 that the Pettis and the McShane integrability coincide (see \cite[Proposition 3.1]{CDPMS1} and  \cite{dp-p,D2010}  and the references inside, 
for the Banach spaces with such a property).
 \item
Finally, when $X$ is separable,  if $\Gamma$ is Bochner measurable and  $ck(X)$-valued and   if all measurable selections of $\Gamma$ are McShane (resp. Birkoff) integrable then $\Gamma$ is McShane (resp. Birkoff) integrable  \cite[Proposition 3.14]{CDPMS1}. 
 If we assume that $\Gamma$ is $cwk(X)$-valued this is false in general, also under the separability assumption of the space,  as shown  in \cite[Remark 3.15]{CDPMS1}.
\end{enumerate}
\end{remark}
Moreover

\begin{theorem}\label{t3}
{\rm \cite[Theorem 3.3]{CDPMS2}}
Let $X$ be any Banach space and let  $\Gamma:[0,1]\to {cwk(X)}$ be a  scalarly measurable multifunction.    Then the following conditions are equivalent:
\begin{itemize}
\item[$(i)$] \, 
$\Gamma$ is McShane integrable;
\item[$(ii)$] \, 
$\Gamma$ is  Henstock integrable  and $\mathscr{S}{}_H(\Gamma)\subset\mathscr{S}{}_{MS}(\Gamma)$.
\item[$(iii)$] \, 
$\Gamma$ is  Henstock integrable  and $\mathscr{S}{}_H(\Gamma)\subset\mathscr{S}{}_P(\Gamma)$;
\item[$(iv)$] \, 
$\Gamma$ is  Henstock integrable  and $\mathscr{S}{}_P(\Gamma)\neq \emptyset$.
\item[$(v)$] \, 
$\Gamma$ is Henstock and Pettis integrable.
\end{itemize}
\end{theorem}
This result was also given  in \cite[Theorem 3.4]{DPM2014},  for $cwk(X)$-valued multifunctions but with $ck(X)$-valued integrals.\\

\section{The variational integrals}\label{s6}
Another valuable tool for studying the integrability of a multifunction is the variational measure associated to its primitive.
In fact it is  important and useful  for example  in  differential inclusions,
\cite{A2005,CCS,C2004,PMS,G2022,M2025,Mesquita2025,Satco2008,Satco2013,Sik-2007,F2022}.
 Therefore  a variationally gauge integrability was also considered. 
 The variational case is more complicated because the 
R{\aa}dstr\"{o}m isometry works with the norm outside the sum.  We  first introduce the variational measure generated by a finitely additive  interval multimeasure.\\

We recall that

\begin{definition} \rm An interval map  $\Phi:\mathscr{I} \rightarrow cwk(X)$ is said to be   \textit{finitely additive}, 
if $\Phi(I{}_1 \cup I{}_2)=\Phi(I_1)\, \dot{+} \, \Phi(I_2)$ for every  $I_1, I_2 \in \mathscr{I}$ with $ I^{\circ}_1 \cap I^{\circ}_2= \emptyset$ and  $ I_1 \cup I_2 \in \mathscr{I}$.
In this case $\Phi$ is said to be an {\it interval multimeasure}.
\\
A map $M:\mathscr{L} \rightarrow cwk(X)$ is said to be a {\it  multimeasure} if for every $x{}^*\in X{}^*$, the map
$\mathscr{L}  \ni A\mapsto s(x{}^*,M(A))$ is a real valued measure, 
 \cite[Theorem 8.4.10]{hp}.
\\
A multimeasure $M:\mathscr{L} \rightarrow cwk(X)$ is  said to be
{\it absolutely continuous } with respect to the Lebesgue measure $\lambda$ (in short 
  $\lambda$\textit{-continuous}) and we write $M\ll\lambda$, if $M(A)=\{0\}$ for each $A \in \mathscr{L}$ with $\lambda(A)=0$.
\end{definition}
\begin{remark} \rm
 The indefinite integrals of Henstock  integrable multifunctions are interval multimeasures,
whereas   the indefinite integrals of Pettis (hence also McShane  or Birkhoff) integrable multifunctions are multimeasures (see \cite[Remark 2.7]{CDPMS2}).
\end{remark}

\begin{definition}\label{vma} \rm
The \textit{variational measure $V_\Phi: \mathscr{L} \rightarrow \mathbb{R}$ } generated by an interval multimeasure $\Phi :\mathscr{I} \rightarrow cwk(X)$
is defined by
	\[V_\Phi(E):=\inf_{\gamma}\left\{Var(\Phi,\gamma,E):\gamma \ \text{is a gauge on }E\right\},\]
where
\[Var(\Phi, \gamma,E)=\sup
\left\{
\sum_{j=1}^p\|\Phi(I_j)\|\colon
\{(I_j,t_j)\}_{j=1}^p\in\Pi_{\gamma}^P\;\mbox{ and } t_j \in E, j=1,\dots,p.
\right\}\]
\end{definition}

Let $\Gamma$ be a 
 Henstock integrable multifunction  and 
 let $f \in \mathscr{S}_H(\Gamma)$.
 Their indefinite integrals $\Phi_{\Gamma}$ and $\Phi_f$, and the 
 variational measures generated by $\Phi_{\Gamma}$  and $\Phi_f$ are related in the following way, see \cite[Proposition 2.7]{CDPMS1}:
\[
{\scriptstyle (H)}\int_I f d\lambda \in {\scriptstyle (H)}\int_I \Gamma d\lambda \quad\mbox{and}\quad V_{\Phi_f}(I)\leq{V_{\Phi_{\Gamma}}(I)}\,.\]
Note that the previous result is also applicable to McShane and Birkhoff integrable multifunctions.
 In these cases, the conclusion is valid not only  for intervals, but also for any arbitrarily measurable subsets of  $\mathscr{L}$.

\begin{definition} \rm
   A multifunction $\Gamma:[0,1]\to cwk(X)$ is said to be {\it variationally Henstock 
   integrable},
   if there exists an interval multimeasure  $\varPhi_{\Gamma}: \mathscr{I} \to {cwk(X)}$ with the following property:
   for every $\varepsilon>0$ there exists a gauge $\gamma$
   on $[0,1]$ such that for each
   $\{(I_1,t_1), \dots,(I_p,t_p)\}\in\Pi_{\gamma}^P$ 
we have:
\begin{eqnarray}\label{aa}
\sum_{j=1}^pd_H \left(\varPhi_{\Gamma}(I_j),\Gamma(t_j)|I_j|\right)<\varepsilon\,.
\end{eqnarray}
   We write then ${\scriptstyle (vH)}\int_0^1\Gamma\,d\lambda:=\varPhi_{\Gamma}([0,1])$.
 The set  multifunction  $\varPhi_{\Gamma}$  will be  called the {\it variational Henstock}  {\it primitive} of $\Gamma$.
The variational integral on a set $I \in \mathscr{I}$ can be defined in an analogous way.
   \end{definition}
   
   If in the previous definition the partitions  belong to $\Pi_{\gamma}$, we obtain the {\it variationally McShane
  integral}.
    
\begin{remark}\label{rem6.5} \rm 
\phantom{a}
\begin{enumerate}[label=\alph*)]
\item 
 When a multifunction reduces to a function \( f: [0,1] \to X \), the set \( \varPhi_f([0,1]) \) simplifies to a single vector in \( X \), and the above definition aligns with those for vector-valued functions.
\item 
By the definitions it follows that if $\Gamma$ is a variationally McShane integrable multifunction, then $\Gamma$ is variationally Henstock, Henstock and also Henstock-Kurzweil-Pettis integrable.

\item   We observe that if $\Gamma$ is a $cwk(X)$-valued McShane integrable multifunction then $\Gamma$ is variationally Henstock integrable if and only if $\Gamma$ is Bochner measurable and the variational measure associated
 with  ${\scriptstyle (MS)} \int_{\cdot} \Gamma\, d\lambda$ is $\lambda$-continuous, \cite[Theorem 3.3]{CDPMS4}. 
 \item Moreover if $\Gamma$ is variationally Henstock integrable and  $0\in \Gamma(t)$ a.e.,
then $\Gamma$ is Birkhoff integrable, \cite[Proposition 4.1]{CDPMS1}.
 \end{enumerate}
   \end{remark}

We now adress  the question of whether  there is at least a selection 
of a 
$cwk(X)$-valued variationally Henstock integrable multifunction which is 
 variationally Henstock integrable too.
The exhistence result is based on the  fact
that, 
when $X$ is an any Banach space, then  each scalarly measurable selection of a Pettis (resp. HKP)
 integrable multifunction $\Gamma:[0,1]\to{cwk(X)}$ is also Pettis (resp. Henstock-Kurzweil-Pettis) integrable (see Cascales, Kadets and Rodr\'iguez \cite{CKR2009}, Di Piazza and Musia\strokel  \cite{DPM2009} and 
  Musia\strokel \cite{M2011}).
The  result we quote here extends that given in  \cite[Theorem 3.12]{CDPMS1}.
\begin{theorem}\label{T4.1}
{\rm \cite[Theorem 5.1]{CDPMS2}}
Let $X$ be any Banach space and let   $\Gamma:[0,1]\to cwk(X)$ be a variationally Henstock integrable multifunction. Then  ${\mathcal{S}}_{vH} \neq \emptyset$ and every strongly
measurable selection of \, $\Gamma$ is also variationally Henstock integrable.
\end{theorem}
\begin{proof}
 As
$\Gamma$ is variationally Henstock integrable, 
$\Gamma$ is also Bochner measurable (\cite[Proposition 2.8]{CDPMS1}).
So it admits strongly measurable selections, as established in \cite[Proposition 3.3]{CDPMS1}.\\
Let  $f$ be a strongly measurable selection of $\Gamma$. Since 
$f$ is HKP integrable, the multifunction 
$G$ defined as $G:=\Gamma-f$
 is Pettis integrable, as showed in \cite[Theorem 1]{DPM2009}. 
 The function 
$i \circ G$, which is  the difference between 
$i \circ \Gamma$ and $i(\{f\})$, is strongly measurable, together as 
$G$. 
Therefore,  $G$ has an essentially  $d_H$-separable range. This means that there exists a set 
$E\in \mathcal{L}$, with 
$\lambda( [0,1] \setminus E)=0$, 
such that $G(E)$ is $d_H$-separable, implying that
$i \circ G$ is Pettis integrable (see \cite[Theorem 3.4, Lemma 3.3]{cr2005}).
\\
Furthermore, since 
$\Gamma$  is variationally Henstock integrable,
 the  variational measure 
$V_{\Phi}$ associated to the indefinite integral 
of $\Gamma$
  is absolutely continuous. If 
$V_{\phi}$  is the variational measure corresponding to the HKP integral of 
$f$,  then we have  $V_{\phi}\leq V_{\Phi}$, and therefore 
$V_{\phi}$ is also absolutely continuous with respect to  $\lambda$.
 As  $\|G\|\leq \|\Gamma\|+\|f\|$, it follows that 
$V_G$  is  $\lambda$-continuous.
\\
Thus, $i \circ G$ satisfies the conditions of
\cite[Corollary 4.1]{BDpM2} which implies that $i \circ G$ is variationally Henstock integrable. Since $i(\{f\})$ is the difference between $i \circ \Gamma$ and $i \circ G$  it follows that $f$ is variationally Henstock integrable.
\end{proof}

As a byproduct of the above result we can compare the  gauges integrals and the variational one.

\begin{theorem}\label{4.3di1}
{\rm \cite[Theorem 5.5]{CDPMS2}}
Let  $X$ be any Banach space and let   $\Gamma:[0,1]\to {cwk(X)}$ be a variationally Henstock integrable multifunction.
Then    the following conditions are equivalent:
\begin{itemize}
\item[$(i)$] \, ${\mathcal{S}}_{vH}(\Gamma)\subset {\mathcal{S}}_{MS}(\Gamma)$;
\item[$(ii)$] \, ${\mathcal{S}}_{vH}(\Gamma)\subset {\mathcal{S}}_P(\Gamma)$;
\item[$(iii)$] \, ${\mathcal{S}}_P(\Gamma)\neq \emptyset$;
\item[$(iv)$] \, $\Gamma$ is   Pettis integrable;
\item[$(v)$] \,  $\Gamma$ is McShane integrable.
\end{itemize}
\end{theorem}
\begin{remark}\rm 
In  \cite[Theorem 4.3]{CDPMS1}  the result had an additional hypothesis,  $X$ was assumed to have  the RNP property which ensures the existence of variationally Hestock integrable selections for $\Gamma$.
This  additional hypothesis
was removed thanks to the Theorem \ref{T4.1}. Moreover this theorem generalizes \cite[Theorem 3.4]{DPM2014}, which proved it  for $cwk(X)$-valued multifunctions with compact valued integrals.
\end{remark}

{\small 
{\bf Acknowledgement and funding.}
This study was partly funded by the Unione europea - Next Generation EU,   Research project of MUR 
PRIN 2022 PNRR: RETINA  “REmote sensing daTa INversion with multivariate functional modeling for essential climAte variables characterization” - Project code P20229SH29; CUP J53D23015950001.
\\
The authors are members of the ``Gruppo Nazionale per l'Analisi Matematica, la Probabilità e le loro Applicazioni'' (GNAMPA) of the Istituto Nazionale di Alta Matematica (INdAM) and of GRUPPO DI LAVORO UMI - Teoria dell'Approssimazione e Applicazioni - T.A.A. The lasut author is a membre of ``Centro di Ricerca Interdipartimentale Lamberto Cesari'' of the University of Perugia.
}

\Addresses
\end{document}